\documentclass[reqno]{amsart}
\numberwithin{equation}{section}

\title{$L^p$ maximal estimates for Weyl sums with $k\ge3$ on $\T$}
\author{}
\date{}

\usepackage{mathtools}
\usepackage{indentfirst}
\usepackage{amsmath}
\usepackage{amssymb}
\usepackage{amsthm}
\usepackage{hyperref}
\usepackage{color}

\usepackage{bm}

\usepackage{framed}

\newcommand{\R}{\mathbb{R}}
\newcommand{\Z}{\mathbb{Z}}
\newcommand{\N}{\mathbb{N}}

\newcommand{\T}{\mathbb{T}}

\newcommand{\al}{\alpha}
\newcommand{\ndiv}{\nmid}
\newcommand{\Acal}{\mathcal{A}}

\begin{document}

\author{Xuezhi Chen}
\address{Institute of Applied Physics and Computational Mathematics, Beijing 100088, China}
\email{xuezhi-chen@foxmail.com}

\author{Changxing Miao}
\address{Institute of Applied Physics and Computational Mathematics, Beijing 100088, China}
\email{miao\_changxing@iapcm.ac.cn}

\author{Jiye Yuan}
\address{School of Science, China University of Geosciences (Beijing), Beijing 100083, China}
\email{yuan\_jiye@cugb.edu.cn}

\author{Tengfei Zhao}
\address{School of Mathematics and Physics, University of Science and Technology Beijing, Beijing 100083, China}
\email{zhao\_tengfei@ustb.edu.cn}

\maketitle

\vspace{-.5cm}

\begin{abstract}
  In this paper, 
  we study the $L^p$ maximal estimates for the Weyl sums $\sum_{n=1}^{N}e^{2\pi i(nx + n^{k}t)}$ with higher-order $k\ge3$ on $\T$, and obtain the positive and negative results. Especially for the case $k=3$, our result is sharp up to the endpoint. The main idea is to investigate the structure of the set where large values of Weyl sums are achieved by making use of the rational approximation and the refined estimate for the exponential sums.
\end{abstract}

\vspace{0.5cm}

\begin{center}
 \begin{minipage}{100mm}
 \end{minipage}
 \end{center}

\newtheorem{thm}{Theorem}[section]
\newtheorem{lemma}[thm]{Lemma}
\newtheorem{cor}[thm]{Corollary}
\newtheorem{prop}[thm]{Propostion}
\newtheorem{remark}[thm]{Remark}
\newtheorem{defi}[thm]{Definition}
\newtheorem{conj}[thm]{Conjecture}
\renewcommand\theequation{\thesection.\arabic{equation}}

\newenvironment{equ}{\begin{equation}}{\end{equation}}
\newenvironment{equ*}{\begin{equation*}}{\end{equation*}}

\section{Introduction}\label{Introduction}

In recent years, the pointwise convergence problem 
 for the solutions of the free dispersive equations  has been widely studied. This problem was first posed by Carleson \cite{Carleson1980} for the free Schr\"{o}dinger equation in the Euclidean spaces in 1980, which asks for the minimal $s$ such that for any initial data $f(x)\in H^{s}(\R^d)$, it holds that the free solution $e^{-it\Delta}f(x)$ converges to
$f(x)$  almost everywhere as $t$ tends to zero. And he obtained that $s\ge\frac{1}{4}$ is a sufficient condition  for the one-dimensional case by reducing the pointwise convergence problem to the maximal estimate for the Schr\"{o}dinger operator $e^{-it\Delta}$, which becomes the general method to study this problem later. Then in 1981, Dahlberg and Kenig \cite{DK} proved that the result obtained by Carleson on $\R$ is actually sharp. For the dispersive operator $e^{it(-\Delta)^{\alpha}}$ with $\alpha>1$ on $\mathbb{R}$, the corresponding sharp results have been obtained by Sj\"olin \cite{Sjolin} in 1987. Recently, Bourgain \cite{Bourgain2016} obtained the necessary condition $s\ge\frac{d}{2(d+1)}$ for the Schr\"{o}dinger operator and the dimension $d\ge1$.  Du-Guth-Li in \cite{DGL} and Du-Zhang in \cite{DZ}  proved that the necessary condition due to Bourgain is also sufficient up to the endpoint in dimensions $d=2$ and $d\ge3$, respectively.

In the periodic case, the solution of
the  free dispersive equation on $\T$
\begin{equation}\label{Schrodinger-equation}
  \begin{cases}
    \frac{i}{2\pi}\partial_{t}u - \left(\frac{1}{2\pi i}{\partial_{x}}\right)^{k}u = 0,\quad (x,t)\in\T^{2},\\
    u(x,0) = u_0(x)\, \in H^s(\mathbb{T})
  \end{cases}
\end{equation}
is given by
\begin{equ}\label{schrodinger-sol}
    u(x,t)=\sum_{n}\widehat{u_0}(n)e^{2\pi i(n\cdot x+n^{k}t)}.
\end{equ}
It is an interesting question to  search for the optimal $s$ such that $\lim_{t\to 0}u(x,t)=u_0(x)$ almost everywhere whenever $u_0\in H^{s}(\mathbb{T})$. Similar to the Euclidean case, the strategy for studying the pointwise convergence problem in the periodic case is also to reduce it to  the  estimate of the corresponding maximal operator. In other words, we need  an estimate like
\begin{equation}\label{max-estimate}
  \Big\|\sup_{0<t<1}\big|\sum_{|n|\le N}a_{n}e^{2\pi i(n\cdot x+n^{k}t)}\big|\Big\|_{L^{p}(\T)}
  \lesssim N^{s_{p}}\Big(\sum_{|n|\le N}|a_{n}|^{2}\Big)^{\frac{1}{2}},
\end{equation}
for some $s_p$. Using Strichartz estimates, Moyua and Vega in \cite{MV}  proved that \eqref{max-estimate} holds for $k=2$, $1\le p\le 6$ and $s_{p}>\frac{1}{3}$. One can see the corresponding results for dimension $d=2$ in \cite{WZ} and for higher dimensions $d\ge3$ in \cite{CLS}.

The results for the periodic case is far from sharpness, and the pointwise convergence problem for the dispersive operator in the periodic case is still open. In the Euclidean case, $e^{-it\Delta}f(x)$ decays to zero as the time $t$ tends to $\infty$, which is the result of the local energy decaying of the solution to the dispersive equation. However in the periodic case, more difficulties in this problem arises, and the reason is that, at least formally, there are more resonances to be controlled, and there are not dispersive estimates and the local energy decaying for the dispersive operator $e^{-it\Delta}$ in the periodic case.

The pointwise convergence problem for the dispersive operator can be reduced to the corresponding maximal estimate, which is related to harmonic analysis, PDEs and number theory.   Consequently, tools from these variant fields can be employed to attack this problem.  

For the case $k=2$, Barron in \cite{Barron} and Baker  in \cite{Baker2021} used different methods to deal with the $L^p$ maximal estimate of the quadratic Weyl sums $\sum_{n=1}^{N}e^{2\pi i(nx+n^{2}t)}$ which corresponds  to \eqref{max-estimate} in the special case where $k=2$ and $a_{n}=1$ for all $1\le n\le N$.
And they obtained the sharp result $s>\frac{1}{4}$. And the corresponding sharp results of the higher dimension cases is due to the latter three authors and Barron \cite{MYZ} (with an appendix by Alex Barron). The maximal estimates for the Weyl sums
 are of interest in their own right. The corresponding results reflect the constructive interference and the structure of the set on which Weyl sums obtain large values. And its application to the Talbot effect is related to the research for the diffraction of light in physics. Please see \cite{ET} for more details about the Talbot effect.

In this paper, we devote to this problem in the higher order cases $k\ge3$ on $\T$.  Compared to the results for $k=2$ in \cite{Baker2021} and \cite{Barron},
 the maximal estimates of the Weyl sums for $k\ge3$ are also difficult to obtain, due to the influence of the different geometric surface reflected by the parameter $k$.  The following is the $L^{p}$ maximal
estimates for Weyl sums.

\begin{thm}\label{refinement-max-estimate}
If the integer $k\ge3$, then
  \begin{equ}\label{max-estimate-2}
  \bigg\|\sup_{0<t<1}\Big|\sum_{n=1}^{N}e^{2\pi i(nx + n^{k}t)}\Big|\bigg\|_{L^{p}(\T)} \lesssim N^{\frac{1}{2}+s_{p,k}+\epsilon}.
  \end{equ}
where for $k\ge3$
\begin{equation}\label{s_p,k}
  s_{p,k} =
 \begin{cases}
 \frac{1}{2}-\frac{1}{p_{k}}, \quad &\textup{if}\,\,\,1\le p\le p_{k},\\
 \frac{1}{2}-\frac{1}{p}, \quad &\textup{if}\,\,\,p>p_{k},
 \end{cases}
\end{equation}
where $p_{k} = \min\{2^{k-1},k(k-1)\}$.
\end{thm}

\begin{remark}
(1) For the cases $k\ge3$, we establish 
   the maximal estimate \eqref{max-estimate-2} by investigating the structure of the set on which Weyl sums obtain large values and then getting the super-level-set  estimates, which play a central role in obtaining Theorem \ref{refinement-max-estimate}.

(2) Besides the results and methods for the case $k=2$ in \cite{Barron}, in order to obtain the $L^p$ estimate for Weyl sums for the cases $k\ge3$, we also utilize tools associated with 
  the power-free and power-full properties of the integers, which can largely improve the results and is the key ingredient to get the sharp result for the case $k=3$.
\end{remark}

\begin{thm}[Negative result]\label{negative-theorem}
For any integer $k\ge3$, there holds for $N$ large enough
\begin{equation}\label{lower-max-estimate}
  \bigg\|\sup_{0<t<1}\Big|\sum_{n=1}^{N}e^{2\pi i(nx + n^{k}t)}\Big|\bigg\|_{L^{p}(\T)} \gtrsim N^{\frac{1}{2}+\gamma_{p}}
\end{equation}
where
\begin{equation}
  \gamma_{p} =
  \begin{cases}
    \frac{1}{4},\quad&\textup{if}\,\,1\le p\le 4,\\
    \frac{1}{2}-\frac{1}{p}, \quad&\textup{if}\,\,p>4.
  \end{cases}
\end{equation} 
\end{thm}

\begin{remark}
(1) For the cases $k\ge3$, we construct a nontrivial counterexample which gives the lower bound of the maximal estimates \eqref{lower-max-estimate}, which coincides with that in \cite{Baker2021} and \cite{Barron} for the case $k=2$.
     We need the estimates of the general Gauss sums  $\sum_{n=1}^{q}e^{2\pi i(\frac{a}{q}n^{k}+\frac{b}{q}n)}$ for higher-order $k\ge3$ and count the number of the intervals for the set where Weyl sums obtain the lower value $\sqrt{q}$ using pigeonholing and the prime number theorem.

(2) By Theorem \ref{refinement-max-estimate} and Theorem \ref{negative-theorem}, we can see that the result of the maximal estimate for the case  $k=3$ is sharp up to the endpoint. However for the higher-order cases $k\ge4$, this problem is left to be studied further to see that if more powerful tools than the Weyl's differencing method and the decoupling method can be applied to the maximal estimates of Weyl sums.

(3) By Theorem \ref{negative-theorem}, we see that the lower bound obtained now is independent of the parameter $k$. It is unknown whether the lower bound is sharp or not.
\end{remark}

\begin{remark}
    Note that  the possible values of $r_k$ in \eqref{s_p,k}, namely $2^{k-1}$ and $k(k-1)$,  are closely related to the pointwise estimate of the general Weyl sums. The value $2^{k-1}$ is obtained by Weyl's differencing method, while the value $k(k-1)$ is obtained by Vinogradov's mean value theorem established by Bourgain, Demeter and Guth \cite{BDG} and Wooley \cite{Wooley2017}. It is conjectured that for any fixed integer $k\geq 2$ and positive number $\epsilon>0$, if $P(x)=\sum_{j=1}^k\alpha_j x^j$ with $|\alpha_k-a/q|\leq 1/q^2$ and $(a,q)=1$, then
    \begin{equ}
        \sum_{n=1}^N e\left(P(n)\right)\lesssim_{k,\epsilon} N^{1+\epsilon}\left(\frac{1}{q}+\frac{q}{N^k}\right)^{1/k}.
    \end{equ}
   As illustrated in Appendix \ref{appendix-a}, if we assume the conjecture is valid, we could take $r_k=k+1$ in \eqref{s_p,k}, which is an potential improvement of Theorem \ref{refinement-max-estimate}  but still far from the lower bound given in Theorem \ref{negative-theorem} when $k\geq 4$.
\end{remark}

%

This paper is organized as follows. In Section \ref{Preliminary}, we give some preliminary tools and
lemmas. In Section \ref{positive-refinement}, we give the proof of Theorem \ref{refinement-max-estimate}.
In Section \ref{negative-result}, we construct
the counterexample of the maximal estimates \eqref{max-estimate-2}, which gives the proof of
Theorem \ref{negative-theorem}. And in Appendix \ref{appendix-a}, we sketch a potential improvement of Theorem \ref{refinement-max-estimate} assuming the Weyl sums estimate.

\section{Preliminaries}\label{Preliminary}

In this section, we present a series of lemmas that are instrumental in characterising the structure of the set of points for which the Weyl sum assumes a significant value.


For $\bm{u}=(u_1,u_2,\ldots,u_k)$, define
\begin{equation}\label{S-k}
  S_{k}(\bm{u};N) := \sum_{n=1}^{N}e(u_{1}n + u_{2}n^{2} + \cdots + u_{k}n^{k}),
\end{equation}
where $e(x):=e^{2\pi ix}$.
The following lemma gives an approximation of the Weyl sums.

\begin{lemma}[Vaughan \cite{Vaughan}, Theorem 7.2
]\label{decom-S-d}
  Suppose that for integers $q$, $r_{1}$, $r_{2}$, $\ldots$, $r_{k}$ satisfying
  \begin{equ*}
    \gcd(q,r_{1},r_{2},\cdots,r_{k}) = 1,
  \end{equ*}
  we have
  \begin{equation}
    |u_{j}-\tfrac{r_{j}}{q}|\le \xi_{j}, \quad j=1,2,\cdots,k.
  \end{equation}
  for some $\xi_{1}, \xi_{2}, \cdots \xi_{k}\in\R$. Then
  \begin{equation}\label{decomposition-S-d}
    S_{k}(\bm{u},N) = q^{-1}S_{k}(q^{-1}\bm{r};q)I(\bm{\xi}) + \Delta.
  \end{equation}
  where $\bm{r} = (r_{1},r_{2},\cdots,r_{k})$, $\bm{\xi}=(\xi_{1},\xi_{2},\cdots,\xi_{k})$,
  \begin{equ}
    I(\bm{\xi}) = \int_{0}^{N}e(\xi_{1}z+\xi_{2}z^{2}+\cdots+\xi_{k}z^{k})\,\mathrm{d}z
  \end{equ}
  and $\Delta$ satisfies the bound
  \begin{equation}\label{delta-bound}
    \Delta \lesssim q(1+|\xi_{1}|N+|\xi_{2}|N^{2}+\cdots+|\xi_{k}|N^{k}).
  \end{equation}
\end{lemma}

\begin{remark}
  With the same notations as in Lemma \ref{decom-S-d}, \cite[Theorem 7.3]{Vaughan} tells us that the auxiliary
  function $I(\xi)$ satisfies
  \begin{equation}
    I(\xi) \lesssim N(1+\xi_{1}N+\cdots+\xi_{k}N^{k})^{-\frac{1}{k}},
  \end{equation}
  which yields that
  \begin{equation}\label{I-bound}
  I(\xi)\lesssim N \min_{j=1,\cdots,k}\{1,\xi_{j}^{-\frac{1}{k}}N^{-\frac{j}{k}}\}.
  \end{equation}
\end{remark}

Let
\begin{equation}\label{weyl-sums-k}
\omega_{N,k}(x,t) = \sum_{n=1}^{N}e^{2\pi i(nx + n^{k}t)}
\end{equation}
for $x\in\T$. It is easy to see that
\begin{equation*}
\omega_{N,k}(x,t) = S_{k}(\bm{u};N) 
\end{equation*}
with $u_{1} =x$, $u_{j} = 0$ for $2\le j\le k-1$ and $u_{k} = t$.


To describe the structure of the set where large values of Weyl sums are achieved, we combine Theorem 4 in \cite{Baker2016} with Theorem 3 in \cite{Baker1982} to obtain the following result. For the strategy of the proof, one can see Appendix \ref{appendix-a} and the reference therein. 

\begin{prop}\label{stru-3}
Let $k\ge3$ and some $\epsilon>0$. Suppose that
$$A>N^{1-\frac{1}{L}+\epsilon},$$
where
\begin{equation}
  L=\min\{2^{k-1},k(k-1)\},
\end{equation}
we have $|\omega_{N,k}(x,t)|\ge A$. Then there exist integers $q,r_1,r_k$ such that
\begin{equation}
  1\le q\le (NA^{-1})^{k}N^{\epsilon}, \quad \gcd(q,r_{1},r_{k}) = 1,
\end{equation}
and
\begin{align}
  &|x-\tfrac{r_1}{q}|\le q^{-1}(NA^{-1})^{k}N^{-1+\epsilon},\\
  &|t-\tfrac{r_k}{q}|\le q^{-1}(NA^{-1})^{k}N^{-k+\epsilon}.
\end{align}

\end{prop}

For the case $k=2$ of \eqref{weyl-sums-k}, Baker in \cite{Baker2021} proved the sharp maximal estimates for Weyl sums using similar results as in Proposition \ref{stru-3}
\begin{equation*}
  \bigg\|\sup_{0<t<1}\Big|\omega_{N,2}(x,t)\Big|\bigg\|^{p}_{L^{p}(\T)} \lesssim N^{a(p)}(\log N)^{b(p)}
\end{equation*}
where
\begin{equation*}
  a(p) = \begin{cases}
    \frac{3p}{4}, &(1\le p\le4)\\
    p-1,&(p>4)
  \end{cases}
  \quad \text{and}
  \quad
  b(p) = \begin{cases}
    1, &(p=4)\\
    0, &(p\ge1, p\neq 4).
  \end{cases}
\end{equation*}
However, to prove Theorem \ref{refinement-max-estimate}, Proposition \ref{stru-3} is not enough to give the estimate \eqref{max-estimate-2}. 
We also need some refined result (see \cite{BCS}) for the Gauss sums to improve Proposition \ref{stru-3}
as follows.


\begin{defi}[\cite{ES}, \cite{Mirsky}]\label{power-full-free}
An integer number $m$ is called

\begin{enumerate}
  \item[$\bullet$] {\bf $r$-th power free} if any prime number $p\mid m$ satisfies $p^{r}\nmid m$;

  \item[$\bullet$] {\bf $r$-th power full} if any prime number $p\mid m$ satisfies $p^{r}\mid m$.
\end{enumerate}

We note that 1 is both $r$-th power free and $r$-th power full for any $r\in\N$.

\end{defi}

For the numbers which are power-full, we recall the following result which shows the density of these numbers in the interval $[1,x]$ with $x\in\N^{+}$.

\begin{lemma}[\cite{ES}]\label{F-number}
  For any integer $i\ge2$, denote
  \begin{equ}
    \mathcal{F}_{i}=\{n\in\N: n\,\,\text{is}\,\, i\!-\!\text{th power full}\} \,\,\,\,\text{and} \,\,\,\,
    \mathcal{F}_{i}(x) = \mathcal{F}_{i}\cap[1,x].
  \end{equ}
    Then we have
    \begin{equ}
      \#\mathcal{F}_{i}(x) \lesssim x^{\frac{1}{i}}.
    \end{equ}
\end{lemma}

Next, in order to estimate the Lebesgue measure of the set where large values of Weyl sums are achieved,
we recall two auxiliary lemmas.

The general bound $q^{1-\frac{1}{k}}$ for the Gauss sum $S_{k,q}(\bm{b})=S_{k}(q^{-1}\bm{b};q)$
defined by \eqref{S-k} is not enough to give the maximal estimate \eqref{refinement-max-estimate}, thus
we utilize the new tools, that is associated to the notion for power-full and power-free, and give a new estimate
of the Gauss sum $S_{k,q}(\bm{b})$ as follows.

\begin{lemma}[\cite{BCS}]\label{refinement-Gauss-sum}
  Write an integer $q\ge1$ as $q=q_{2}\cdots q_{k}$ with $\gcd(q_{i},q_{j})=1$, $2\le i<j\le k$, such that
  \begin{enumerate}
    \item[$\bullet$] $q_{2}\ge1$ is cube free,

    \item[$\bullet$] $q_{i}$ is $i$-th power full but $(i+1)$-th power free when $3\le i\le k-1$,

    \item[$\bullet$] $q_{k}$ is $k$-th power full.
  \end{enumerate}
  For $\bm{b}\in\Z^k$ with
  \begin{equation}
    \gcd(q,b_{1},\cdots,b_{k}) = 1,
  \end{equation}
  we have for $\epsilon>0$
  \begin{equation}
    |S_{k,q}(\bm{b})|\le\prod_{j=2}^{k}q_{j}^{1-\frac{1}{j}}q^{\epsilon}.
  \end{equation}
\end{lemma}



From the estimate for $S_{k,q}(\bm{b})$ in Lemma \ref{refinement-Gauss-sum}, we will see this estimate improves the general bound $q^{1-\frac{1}{k}}$, which will give an improved estimate for the Lebesgue measure of the set related to the large values of Weyl sums $\omega_{N,k}(x,t)$.

Using Lemma \ref{refinement-Gauss-sum},
we can improve the result of Proposition \ref{stru-3}.

\begin{prop}\label{refinement-structure}
  For $k\ge3$ and some $\epsilon>0$, suppose that
  \begin{equation}\label{lower-cons}
    A>N^{1-\frac{1}{D}+\epsilon},
  \end{equation}
  where
  \begin{equation}
    D=\min\{2^{k-1},k(k-1)\}.
  \end{equation}
  If $|\omega_{N,k}(x,t)|\ge A$, then there exist positive integers $q_{2}, \cdots,  q_{k}$ with
  $\gcd(q_{i},q_{j})=1$, $2\le i<j\le k$, such that
  \begin{enumerate}
    \item[$\bullet$] $q_{2}\ge1$ is cube free,

    \item[$\bullet$] $q_{i}$ is $i$-th power full but $(i+1)$-th power free when $3\le i\le k-1$,

    \item[$\bullet$] $q_{k}$ is $k$-th power full,
  \end{enumerate}
  and
  \begin{equation}
    \prod_{j=2}^{k}q_{j}^{\frac{1}{j}} \le N^{1+\epsilon}A^{-1}
  \end{equation}
  and integers $r_{1}, r_{k}$ with
  \begin{equation}
    \gcd(q_{2}\cdots q_{k},r_{1},r_{k})=1
  \end{equation}
  such that
  \begin{align}
    &|x - \tfrac{r_{1}}{q_{2}\cdots q_{k}}|\le (NA^{-1})^{k}N^{-1+\epsilon}
    \prod_{j=2}^{k}q_{j}^{-\frac{k}{j}},\label{x-interval}\\
    &|t - \tfrac{r_{k}}{q_{2}\cdots q_{k}}|\le (NA^{-1})^{k}N^{-k+\epsilon}
    \prod_{j=2}^{k}q_{j}^{-\frac{k}{j}}.\label{t-interval}
  \end{align}
\end{prop}

\begin{proof}

Since $|\omega_{N,k}(x,t)|\ge A$, by Proposition \ref{stru-3}, then for $\epsilon>0$ with \eqref{lower-cons} there exist integers $q,r_1,r_{k}$ such that
\begin{equation}
  1\le q\le (NA^{-1})^{k}N^{\epsilon}, \quad \gcd(q,r_{1},r_{k}) = 1,
\end{equation}
and
\begin{align}
  &|x-\tfrac{r_1}{q}|\le q^{-1}(NA^{-1})^{k}N^{\epsilon},\\
  &|t-\tfrac{r_{k}}{q}|\le q^{-1}(NA^{-1})^{k}N^{-k+\epsilon}.
\end{align}

By the approximation of Weyl sums as in Lemma \ref{decom-S-d}
\begin{equation}
  \omega_{N,k}(x,t) = q^{-1}S_{k}(q^{-1}\bm{r};q)I(\bm{\xi}) + \Delta,
\end{equation}
where $\bm{\xi}=(\xi_1,\xi_2,\cdots,\xi_k)$, $S_{k}(q^{-1}\bm{r};q)$ is given by \eqref{S-k}, $I(\bm{\xi})$ and $\Delta$ are defined in Lemma \ref{decom-S-d}
and
\begin{align*}
  &\xi_{1} = x - \tfrac{r_{1}}{q}, \quad \xi_{k} = t - \tfrac{r_{k}}{q},\\
  &\xi_{j} = -\tfrac{r_{j}}{q} \,\,\text{with $r_j=0$  for} \,\,2\le j\le k-1.
\end{align*}

By \eqref{delta-bound}, we have
\begin{equation*}
  |\Delta| \lesssim q + (NA^{-1})^{k}N^{\epsilon} \lesssim (NA^{-1})^{k}N^{\epsilon}\le \tfrac{A}{2},
\end{equation*}
if $N$ is large enough.

Thus by triangle inequality, we have
\begin{equation}\label{approximation-lower-bound}
  \tfrac{A}{2}\le |\omega_{N,k}(x,t)| - |\Delta| \le q^{-1}|S_{k}(q^{-1}\bm{r};q)|\cdot|I(\bm{\xi})|.
\end{equation}
By Lemma \ref{refinement-Gauss-sum}, we have
\begin{equation}\label{S-k-q}
 |S_{k,q}(\bm{r})|\lesssim \prod_{j=2}^{k}q_{j}^{1-\frac{1}{j}},\quad \gcd(q,r_{1},\cdots,r_{k})=1,
\end{equation}
where $q$ satisfies the conditions in Lemma \ref{refinement-Gauss-sum}.

Putting \eqref{S-k-q} into \eqref{approximation-lower-bound}, we have
\begin{equation}
\begin{split}
  A
  &\lesssim q^{-1}\prod_{j=2}^{k}q_{j}^{1-\frac{1}{j}+\epsilon}N
  \cdot\min_{i=1,\cdots,k}\{1,|\xi_{i}|^{-\frac{1}{k}}\cdot N^{-\frac{i}{k}}\}\\
  &\lesssim \prod_{j=2}^{k}q_{j}^{-\frac{1}{j}+\epsilon}N
  \cdot\min_{i=1,\cdots,k}\{1,|\xi_{i}|^{-\frac{1}{k}}\cdot N^{-\frac{i}{k}}\},
\end{split}
\end{equation}
which yields that
\begin{equation}
  \prod_{j=2}^{k}q_{j}^{\frac{1}{j}}\lesssim (NA^{-1})\cdot q^{\epsilon} \lesssim (NA^{-1})N^{\epsilon},
\end{equation}
and
\begin{equation}\label{xi-bound}
  A\lesssim \prod_{j=2}^{k}q_{j}^{-\frac{1}{j}+\epsilon}N\cdot|\xi_{i}|^{-\frac{1}{k}}\cdot N^{-\frac{i}{k}},
  \quad 1\le i\le k.
\end{equation}
From \eqref{xi-bound}, we have
\begin{equation}
  |\xi_{i}| \lesssim (NA^{-1})^{k}N^{-i+\epsilon}\prod_{j=2}^{k} q_{j}^{-\frac{k}{j}}.
\end{equation}
By the definition of $\bm{\xi}$ and the case for $i = 1, k$, we can obtain the results \eqref{x-interval}
and \eqref{t-interval}.

\end{proof}

\begin{remark}
  Proposition \ref{refinement-structure} characterizes the structure of the set more precisely than Proposition \ref{stru-3} on which the Weyl sums obtain large values. Indeed, both Proposition \ref{stru-3} and Proposition \ref{refinement-structure} describe the structure of the set where Weyl sums obtain large values, however, Proposition \ref{refinement-structure} gives the smaller 
  estimate for the length of every intervals in the structure of the corresponding set. Meanwhile the number of the corresponding intervals does not grow too much, which we later will see becomes smaller by using the $i$-th power full and $i$-th power free properties.
\end{remark}

\section{Positive result 
}\label{positive-refinement}

For the proof of Theorem \ref{refinement-max-estimate}, we use the following lemma which builds the relationship between the $L^{p}$ norm of the integrable function and its level set estimate.

\begin{lemma}[\cite{BCS}, Lemma 3.1]\label{level-max}
  Let $\mathcal{X}$ be a metric space and $\nu$ be a Radon measure on $\mathcal{X}$ with
  $\nu(\mathcal{X})<\infty$. Let $M\le N$ be two positive numbers and $F: \mathcal{X}\rightarrow [0,N]$ be
  a function such that for any $M\le A\le N$,
  \begin{equation}\label{level-set-F}
    \nu(\{x\in\mathcal{X}: F(x)\ge A\})\le N^{a}A^{-b}.
  \end{equation}
  Then for any $p>0$,
  \begin{equation}
    \int_{\mathcal{X}}F(x)^{p}\,\mathrm{d}\nu(x)
    \lesssim \nu(\mathcal{X})M^{p} + N^{a}M^{p-b}\log N + N^{p + a - b}.
  \end{equation}
\end{lemma}

For the convenience of readers, we give the outlines of the proof for Lemma \ref{level-max}.
Split the integral region into two parts according to the large value and small value of $F(x)$, that is $\{x\in\mathcal{X}:F(x)< M\}$ and $\{x\in\mathcal{X}:F(x)\ge M\}$. Then divide the region $\{x\in\mathcal{X}:F(x)\ge M\}$ into $O(\log N)$ regions by the dyadic values of $F(x)$. Taking the super-level-set estimate \eqref{level-set-F} for $F(x)$ into account yields the result in Lemma \ref{level-max}. See \cite{BCS} for the details.

Using Lemma \ref{level-max}, we see that to prove Theorem \ref{refinement-max-estimate} and obtain the maximal estimate for Weyl sums \eqref{max-estimate-2}, it is sufficient to
estimate the Lebesgue measure of the set where large values of Weyl sums $\omega_{N,k}(x,t)$ are obtained.



With Proposition \ref{refinement-structure} and Lemma \ref{F-number}, we have the following result.

\begin{lemma}\label{refinement-level-set}
Suppose that for some fixed $\epsilon>0$
\begin{equation}
  A>N^{1-\frac{1}{D}+\epsilon},
\end{equation}
where
\begin{equation}
  D = \min\{2^{k-1}, k(k-1)\}, \quad k\ge3.
\end{equation}
Let
\begin{equation}
  M_{k}(A;N) = \Big\{x\in\T: \sup_{0<t<1}|\omega_{N,k}(x,t)|>A\Big\},
\end{equation}
then
\begin{equation}
  |M_{k}(A;N)| \lesssim N^{k+\epsilon}A^{-(k+1)}.
\end{equation}
\end{lemma}

\begin{proof}
Let $Q=(NA^{-1})^{k}$ and
\begin{equation}
  \mathcal{M}_{q_{2},\cdots, q_{k}}
  = \bigg\{x\in\T: |x-\tfrac{b}{q_{2}\cdots q_{k}}|\lesssim
  QN^{-1+\epsilon}\prod_{j=2}^{k}q_{j}^{-\frac{k}{j}},
  \,\,1\le b\le q_{2}\cdots q_{k}\bigg\}
\end{equation}
with $q_{j}$ for $2\le j\le k$ satisfying the conditions
in Proposition \ref{refinement-structure}.

It is easy to see that
\begin{equation}\label{M-measure}
  |\mathcal{M}_{q_{2},\cdots, q_{k}}|
  \lesssim (q_{2}\cdots q_{k})\cdot QN^{-1+\epsilon}\prod_{j=2}^{k}q_{j}^{-\frac{k}{j}},
\end{equation}
and
\begin{equation}
  M_{k}(A;N) \subset \bigcup_{(q_{2},\cdots,q_{k})\in\Omega}\mathcal{M}_{q_{2},\cdots, q_{k}},
\end{equation}
where
\begin{equation}
  \Omega = \bigg\{(q_{2},\cdots,q_{k})\in\mathbb{N}_{+}^{k-1}: q_{j}\in\mathcal{F}_{j}, 3\le j\le k,
  \prod_{j=2}^{k}q_{j}^{\frac{1}{j}}\lesssim Q^{\frac{1}{k}}N^{\eta}\bigg\}
\end{equation}
with $\eta>0$ small.

By pigeonholing, we have
\begin{equation}\label{pigeonholing}
\begin{split}
  |M_{k}(A;N)| \lesssim& \max\bigg\{\vphantom{\prod_{j=2}^{k}}U(Q_{2},\cdots,Q_{k}): \text{dyadic}\,\,Q_{2},\cdots,Q_{k}\ge1,\\
  &\phantom{\max\{U(Q_{2},\cdots,}\prod_{j=2}^{k}Q_{j}^{\frac{1}{j}}\le Q^{\frac{1}{k}}N^{\eta}\bigg\}(\log N)^{k},
\end{split}
\end{equation}
where
\begin{equation}\label{U-definition}
  U(Q_{2},\cdots,Q_{k}) := \bigg|\bigcup_{\substack{q_{j\sim Q_{j},2\le j\le k}
  \\(q_{2},\cdots,q_{k})\in\Omega}}\mathcal{M}_{q_{2},\cdots,q_{k}}\bigg|.
\end{equation}
From \eqref{M-measure}, \eqref{U-definition} and Lemma \ref{F-number}, we have
\begin{equation}\label{U-estimate}
\begin{split}
  U(Q_{2},\cdots,Q_{k})
  &\lesssim \sum_{\substack{q_{j\sim Q_{j},2\le j\le k}
  \\(q_{2},\cdots,q_{k})\in\Omega}}|\mathcal{M}_{q_{2},\cdots,q_{k}}|\\
  &\lesssim \sum_{\substack{q_{j\sim Q_{j},2\le j\le k}
  \\(q_{2},\cdots,q_{k})\in\Omega}}(q_{2}\cdots q_{k})\cdot
  QN^{-1+\epsilon}\prod_{j=2}^{k}q_{j}^{-\frac{k}{j}}\\
  &\sim \Big(\sum_{q_{2}\sim Q_{2}}q_{2}^{1-\frac{k}{2}}\Big)QN^{-1+\epsilon}\prod_{j=3}^{k}
  \sum_{\substack{q_{j}\sim Q_{j}\\q_{j}\in\mathcal{F}_{j}}}q_{j}^{1-\frac{k}{j}}\\
  &\lesssim QN^{-1+\epsilon}\prod_{j=2}^{k}Q_{j}^{\beta_{j}},
\end{split}
\end{equation}
where
\begin{equation*}
  \beta_{j}=
  \begin{cases}
    1-\frac{k}{j}+1,\quad &\text{if}\,\,j=2,\\
    1-\frac{k}{j}+\frac{1}{j},\quad &\text{if}\,\,3\le j\le k.
  \end{cases}
\end{equation*}
Since $k\ge3$, we can see that $\beta_{j}\le\tfrac{1}{j}$ for $2\le j\le k$.
Thus we have
\begin{equation}\label{Q-j-product}
  \prod_{j=2}^{k}Q_{j}^{\beta_{j}} \le \prod_{j=2}^{k}Q_{j}^{\frac{1}{j}} \lesssim
  Q^{\frac{1}{k}}N^{\eta},
\end{equation}
where the last inequality comes from the definition of $Q_{j}$ in \eqref{pigeonholing}.

Combining \eqref{pigeonholing}, \eqref{U-estimate} with \eqref{Q-j-product}, we can obtain
\begin{equation}
\begin{split}
  |M_{k}(A;N)|
  &\lesssim QN^{-1+\epsilon}\cdot Q^{\frac{1}{k}}N^{\eta}\cdot (\log N)^{k}\\
  &\sim (NA^{-1})^{k}\cdot N^{-1+\epsilon}\cdot (NA^{-1})\cdot N^{\eta}\cdot N^{\epsilon}\\
  &\sim (NA^{-1})^{k+1}\cdot N^{-1+\epsilon}\\
  &\sim N^{k+\epsilon}A^{-(k+1)}.
\end{split}
\end{equation}

\end{proof}

Now we will prove Theorem \ref{refinement-max-estimate}.
\begin{proof}[\bf Proof of Theorem \ref{refinement-max-estimate}]
For $D=\min\{2^{k-1},k(k-1)\}$ with $k\ge3$, then
\begin{equation}
 D =
  \begin{cases}
    2^{k-1},\quad \text{if}\,\,3\le k\le5\\
    k(k-1),\quad \text{if}\,\,k\ge6.
  \end{cases}
\end{equation}
For the calculation of the maximal estimate for Weyl sums \eqref{max-estimate-2},
we only consider 
 the cases $3\le k\le5$, since
the calculation for the cases $k\ge6$ is similar.

Set
\begin{equation}
  a=k,\quad b=k+1,\quad M=N^{1-\frac{1}{2^{k-1}}+\epsilon},
\end{equation}
where $\epsilon$ is taken small enough such that $M<N$.

Applying Proposition \ref{refinement-structure}, Lemma \ref{refinement-level-set} and Lemma \ref{level-max} to
\begin{equation}
  I_{p} = \int_{\T}\sup_{0<t<1}|\omega_{N,k}(x,t)|^{p}\,\mathrm{d}x,
\end{equation}
we obtain
\begin{equation}
  I_{p} \lesssim M^{p} + N^{a}M^{p-b}\log N + N^{p+a-b}.
\end{equation}
\begin{enumerate}
  \item[\bf$1^{\circ}$.] For $1\le p\le b$, we have $N^{a}M^{p-b}\ge N^{p+a-b}$.
  Since for $k\ge3$, there holds
  \begin{equation*}
    1\ge\frac{k+1}{2^{k-1}},
  \end{equation*}
  then $M^{p}\ge N^{a}M^{p-b}$, which yields that
  \begin{equation}
    I_{p} \lesssim M^{p} \sim N^{(1-\frac{1}{2^{k-1}})p + \epsilon}.
  \end{equation}

  \item[\bf$2^{\circ}$.] For $p>b$, we have $N^{a}M^{p-b}<N^{p+a-b}$.
  \begin{enumerate}
    \item For $b<p\le2^{k-1}$, there holds $M^{p}\ge N^{p+a-b}$ which yields that
    \begin{equation}
    I_{p}\lesssim M^{p} \sim N^{(1-\frac{1}{2^{k-1}})p + \epsilon}.
    \end{equation}

    \item For $p>2^{k-1}$, there holds $M^{p}< N^{p+a-b}$ which yields that
    \begin{equation}
    I_{p}\lesssim N^{p+a-b} \sim N^{p-1+\epsilon}.
    \end{equation}
  \end{enumerate}

\end{enumerate}

Hence, we complete the proof of Theorem \ref{refinement-max-estimate}.

\end{proof}

\section{Negative result}\label{negative-result}

We first give a trivial lower bound for the $L^{p}$ maximal estimates of $\omega_{N,k}(x,t)$,
which gives the sharp negative results in Theorem \ref{negative-theorem} for large $p$.

For $\forall\,\,p>1$, taking $E=[0,10^{-6}N^{-1}]$, by 
the constructive interference for $\omega_{N,k}(x,t)$, we have
\begin{align*}
  &\bigg(\int_{\T}\sup_{0<t<1}|\omega_{N,k}(x,t)|^{p}\,\mathrm{d}x\bigg)^{\frac{1}{p}}\\
  \ge&\bigg(\int_{E}\sup_{0<t<10^{-6}N^{-k}}|\omega_{N}(x,t)|^{p}\,\mathrm{d}x\bigg)^{\frac{1}{p}}\\
  \gtrsim &N\cdot |E|^{\frac{1}{p}} \gtrsim N\cdot N^{-\frac{1}{p}} = N^{1 - \frac{1}{p}}.
\end{align*}

\begin{remark}
  From the estimate above and Theorem \ref{refinement-max-estimate}, we can see that when $p\ge\min\{k(k-1), 2^{k-1}\}$ with $k\ge3$, 
  precisely when
  \begin{equation*}
    p \ge \begin{cases}
      2^{k-1}, &\quad\textup{for}\,\,3\le k\le 5,\\
      k(k-1),&\quad\textup{for}\,\,k\ge6,
    \end{cases}
  \end{equation*}
  the parameter $s_{p,k}$ for the maximal estimate \eqref{max-estimate-2} in Theorem \ref{refinement-max-estimate} is almost sharp
  up to the $\epsilon$ loss.
\end{remark}

The simple example above gives the rough lower bound for the maximal estimates of Weyl sums, but is not good enough for this lower bound estimate for small $p$. Thus,
in order to give an example to obtain a better lower bound for \eqref{lower-max-estimate} in Theorem \ref{negative-theorem} for small $p$,
we construct sets whose Lebesgue measure is greater than $O(1)$ and on which $\omega_{N,k}$ has nontrivial lower bound. For this purpose, we give some necessary tools.

First, we recall the Gauss sums.
For $a, q\in\N$ and $b\in\Z$, define
\begin{equation}\label{k-dispersive}
  S_{k}(a,b,q) := \sum_{n=1}^{q}e^{2\pi i(\frac{a}{q}n^{k}+\frac{b}{q}n)}.
\end{equation}
It is easy to see that
\begin{equation*}
  S_{k}(a,b,q) = S_{k}(\tfrac{\bm{u}}{q}; q)
\end{equation*}
with $\bm{u}=(b, 0,\cdots, 0, a)$ which is $k$-dimensional.

\begin{lemma}[\cite{Oh},\cite{Pierce}]\label{Gauss-sum}
Let $a, q\in\N$ and $b\in\Z$ with $(a,q)=1$. Then, the following holds for the Gauss sums:
\begin{enumerate}
  \item[(1)] When $b$ is even,
  \begin{equation}
    |S_{2}(a,b,q)|=
    \begin{cases}
      \sqrt{q},\quad&\text{if}\,\,q\,\,\text{is odd},\\
      0,\quad&\text{if}\,\,q\equiv2 \,(\text{mod}\,\, 4),\\
      \sqrt{2q},\quad&\text{if}\,\,q\equiv0 \,(\text{mod}\,\, 4).
    \end{cases}
  \end{equation}

  \item[(2)] When $b$ is odd,
  \begin{equation}
    |S_{2}(a,b,q)|=
    \begin{cases}
      \sqrt{q},\quad&\text{if}\,\,q\,\,\text{is odd},\\
      \sqrt{2q},\quad&\text{if}\,\,q\equiv2 \,(\text{mod}\,\, 4),\\
      0,\quad&\text{if}\,\,q\equiv0 \,(\text{mod}\,\, 4).
    \end{cases}
  \end{equation}

\end{enumerate}

\end{lemma}


For the case $k=2$ in \eqref{k-dispersive}, 
which is the Schr\"odinger case, using the Gauss sums estimate which is Lemma \ref{Gauss-sum}, Oh constructed a nontrivial example for the $L^{p}$ maximal estimate for $\omega_{N,2}(x,t)$, which gives the lower bound for \eqref{max-estimate-2} with $k=2$. See \cite{Oh} for the detail.


 Now to construct the nontrivial counterexample for Theorem \ref{negative-theorem}, Lemma \ref{Gauss-sum} is not enough to give the counterexample for the higher order cases $k\ge3$, and we give the generalization of Lemma \ref{Gauss-sum} to the high-degree cases,
which can be found in \cite{ACP}. And we will see that the constructing of the counterexample for high-degree cases $k\ge3$ needs $q$ to be of prime, which is different from the case $k=2$.


We now show that a positive proportion of choices for integral coefficients  lead to a complete exponential sum modulo $q$ of size $\gg q^{1/2}$.

\begin{prop}[An, Chu and Pierce \cite{ACP}, Proposition 2.2]\label{prop_sum_big}
 Fix integers $k\ge2$.
For each integer $q$ and tuple $(a_1, a_2)$ for $1\le a_{j}\le q$, $1\le j\le 2$,
there exist constants $0<\al_1<1$ and $0<\al_2<1$ with $\al_2$ depending on $k$, such that for every prime $q \geq 3$ with $q \ndiv k$, at least $\al_2 q^2$ choices of $(a_1,a_2)$ have $|S_{k}(a_1,a_2;q)| \geq \al_1 q^{1/2}$. In fact, one can take $\al_1=1/2$ and $\al_2 = k^{-2}/4$.
\end{prop}

In order to apply Proposition \ref{prop_sum_big} to estimate the lower bound of $L^{p}$ norm for $\sup_{0<t<1}|\omega_{N,k}(x,t)|$, we need the following corollary, which distinguishes the role of the highest-order coefficient.
\begin{cor}\label{cor_T_big}
Fix integer $k\ge2$.
 Specify $\al_1,\al_2$ to be as in  Proposition \ref{prop_sum_big}. For each prime $q$,
let $\Acal(q)$ denote the set of $(a_1,a_2)$ such that $1\le a_1, a_2\le q$ and $|S_{k}(a_1,a_2;q)| \geq \al_1q^{1/2}.$
For each $a_1,$ define the ``good set''
\[G(a_1) := \{ a_{2}\in[1,q]: (a_1,a_2) \in \Acal(q)\}.\]
Suppose that $q\ge 3$ is a prime such that $q \ndiv k$.
Then for at least one choice of $a_1(q)$ depending on $q$,  we have $\#G(a_1(q)) \geq (\al_2/2) q.$

\end{cor}

\begin{proof}
  We prove this corollary by contradiction.
  Assume for any $a_{1}\in[1,q]$, there holds $\#G(a_1) < (\al_2/2) q$.

  Then we see that
  \begin{equation*}
    \#\mathcal{A}(q) \le \sum_{a_{1}=1}^{q}\#G(a_{1}) < (\al_2/2) q^{2}
  \end{equation*}
  which contradicts with the result of Proposition \ref{prop_sum_big}, $\#\mathcal{A}(q)\ge \al_2 q^{2}$. This fact implies the result in this corollary.
\end{proof}


Finally we give the proof of Theorem \ref{negative-theorem}.

\begin{proof}[\bf Proof of Theorem \ref{negative-theorem}]
Let $k\ge3$.
For $a,\,b,\,q\in\N$, such that $q\ge 3$ is prime, $q\ndiv k$, and $1\le a<q$, define
\begin{equation}\label{J-q-a}
  J(q,a) = [\tfrac{a}{q}-\tfrac{1}{100N}, \tfrac{a}{q}+\tfrac{1}{100N}]
\end{equation}
and
\begin{equation}\label{major-arc}
  \mathcal{M}_{k}(q,a,b) = \{(x,t)\in\T^{2}: |x-\tfrac{a}{q}|\le \tfrac{1}{100N},
  |t-\tfrac{b}{q}|\le \tfrac{1}{100N^{k}}\}.
\end{equation}

Since $q$ is a prime, we have $\gcd(q,a,b)=1$ with $1\le a,b<q$.
Let
$$\bm{\beta} := (\beta_1, \beta_2) = \left(x-\frac{a}{q}, t-\frac{b}{q}\right)$$
for $(x,t)\in\mathcal{M}_{k}(q,a,b)$.
Then by Lemma \ref{decom-S-d}, we have
\begin{equation}\label{decom-omega}
  \omega_{N,k}(x,t) = q^{-1}S_{k}(b,a,q)I(\beta_1,\beta_2) + \Delta
\end{equation}
where
\begin{equation}\label{decom-error}
  |\Delta|\lesssim q(1+|\beta_{1}|N+|\beta_2|N^{k}) \lesssim q,
\end{equation}
and
\begin{equation}\label{decom-I}
  |I(\beta_1,\beta_2)|
  = \bigg|\int_{0}^{N}e^{2\pi i(\beta_{1}z + \beta_{2}z^k)}\,\mathrm{d}z\bigg|
  \sim N.
\end{equation}

By Proposition \ref{prop_sum_big},
there exist at least $\frac{1}{4k^{2}}q^{2}$ choices of $(b,a)$ with
$1\le a,b<q$, such that
\begin{equation}\label{lower-Gauss-sum}
  |S_{k}(b,a,q)|  \gtrsim \sqrt{q}.
\end{equation}
Then by Corollary \ref{cor_T_big}, we have at least one choice of $b(q)\in \Z\cup[1,q)$, such that
\begin{equation}\label{number-a}
  \# G(b(q))\ge \tfrac{1}{8k^{2}}q.
\end{equation}
This fact means that for a fixed $b(q)$, there exist at least $\tfrac{1}{8k^{2}}q$ choices of $a\in\N\cup[1,q)$
satisfying \eqref{lower-Gauss-sum}.

Taking prime number $q\in[c_{1}\sqrt{N}, \sqrt{N}]$ 
with $c_{1}$ small,
it is easy to see that $J(q_{1},a_{1})$ and
$J(q_{2},a_{2})$ are disjoint if $\frac{a_{1}}{q_{1}}\neq \frac{a_{2}}{q_{2}}$ for $q_1, q_2 \in[c_{1}\sqrt{N}, \sqrt{N}]$. By \eqref{decom-omega} -
\eqref{lower-Gauss-sum}, for $(x,t)\in\mathcal{M}_{k}(q,a,b)$, we have
\begin{equation}
  |\omega_{N,k}(x,t)|\gtrsim q^{-1}|S_{k}(b,a,q)|\cdot|I(\beta_{1},\beta_2)|
  \gtrsim \tfrac{N}{\sqrt{q}} \sim N^{\frac{3}{4}}.
\end{equation}

{\bf Case 1, $1\le p\le4$.}

\begin{equation}
\begin{split}
  \bigg(\int_{\T}\sup_{0<t<1}|\omega_{N,k}|^{p}\,\mathrm{d}x\bigg)^{\frac{1}{p}}
  \ge& \int_{\T}\sup_{0<t<1}|\omega_{N,k}|\,\mathrm{d}x\\
  \ge& \sum_{\substack{c_{1}\sqrt{N}\le q\le\sqrt{N},\\ q\,\,\text{prime}}}\sum_{a\in G(b(q))}
  \int_{J(q,a)}\sup_{0<t<1}|\omega_{N,k}|\,\mathrm{d}x\\
  \gtrsim& \sum_{\substack{c_{1}\sqrt{N}\le q\le\sqrt{N},\\ q\,\,\text{prime}}}
  \sqrt{N}\cdot N^{\frac{3}{4}}N^{-1} \\
  \gtrsim& \pi(\sqrt{N})\sqrt{N}\cdot N^{\frac{3}{4}}N^{-1}\\
  \gtrsim& N^{\frac{3}{4}},
\end{split}
\end{equation}
where $b(q)$ is chosen to be one fixed number such that \eqref{number-a} holds.

{\bf Case 2, $p> 4$.}

In this case, we can see that
\begin{equation}
  \begin{split}
  \bigg(\int_{\T}\sup_{0<t<1}|\omega_{N,k}|^{p}\,\mathrm{d}x\bigg)^{\frac{1}{p}}
  \ge& \bigg(\int_{[0,10^{-6}N^{-1}]}\sup_{t\in[0,10^{-6}N^{-k}]}
       |\omega_{N,k}|^{p}\,\mathrm{d}x\bigg)^{\frac{1}{p}}\\
  \gtrsim& N^{1-\frac{1}{p}} \ge N^{\frac{3}{4}}.
  \end{split}
\end{equation}

The proof is completed.
\end{proof}

{\bf Acknowledgements.}
This project was suppported by the National Key R\&D program of China: No.2022YFA1005700 and the NSF of China under grant No.12371095.


\appendix
\section{a potential improvement of Theorem \ref{refinement-max-estimate} assuming the Weyl sums estimate}\label{appendix-a}
In this appendix, we will sketch how a better estimate of the Weyl sum (Conjecture \ref{weyl-sum-conj}) can give a more precisely characterization of the set where the Weyl sums obtain significant values. The proof is mainly from Baker \cite[Chapter 4]{Baker1986} and one can also refer to Baker \cite{Baker1982,Baker2016}.
\begin{lemma}[Dirichlet]\label{lem-app-1} Let $\alpha$ be a real number. Then for any $M\geq 1$, there exists a rational number $a/q$ with $\gcd(a,q)=1$ and $1\leq q\leq M$ and
\begin{equ*}
    \Big|\alpha-\frac{a}{q}\Big|\leq\frac{1}{qM}.
\end{equ*}
\end{lemma}

\begin{lemma}\label{lem-app-2}
   Let $N\geq 1$ and $k\geq 2$. Assume that the polynomial $\sum_{j=1}^k\alpha_j x^j$ satisfies
    \begin{equ*}
        \alpha_j=\frac{a_j}{q}+\beta_j \quad\text{and} \quad|\beta_j|\leq \frac{1}{2k^2q}\frac{N}{N^j}, \quad  1\leq j\leq k.
    \end{equ*}
Denote that $e_q(x)=e^{2\pi ix/q}$. Let $d=\gcd(q,a_2,\ldots,a_k)$. Then for any $\epsilon>0$ and any integer $1\leq T\leq N$, we have
    \begin{equ}\label{app-1}
        \sum_{n=1}^T e\Big(\sum_{j=1}^k \alpha_j n^j\Big)=q^{-1}\sum_{n=1}^qe_q\Big(\sum_{j=1}^k a_j n^j\Big)\int_0^T\!e\Big(\sum_{j=1}^k \beta_j x^j\Big)\,\mathrm{d}x+O(d^{\frac{1}{k}}q^{1-\frac{1}{k}+\frac{\epsilon}{2}}),
    \end{equ}
    where
    \begin{equ}\label{app-2}
        \int_0^T\!e\Big(\sum_{j=1}^k \beta_j x^j\Big)\,\mathrm{d}x\lesssim N\min\left\{1,N^{-\frac{1}{k}}|\beta_1|^{-\frac{1}{k}},\ldots,N^{-1}|\beta_k|^{-\frac{1}{k}}\right\},
    \end{equ}
    and
    \begin{equ}
        \label{app-103}
        \sum_{n=1}^qe_q\Big(\sum_{j=1}^k a_j n^j\Big)=O(q^{1-(1/k)+\epsilon}d^{1/k}).
    \end{equ}

\end{lemma}

\begin{conj}\label{weyl-sum-conj}
    For any fixed integer $k\geq 2$ and  number $\epsilon>0$, we have
    \begin{equ}\label{app-3}
        \sum_{n=1}^N e\Big(\sum_{j=1}^k\alpha_j n^j\Big)\lesssim_{k,\epsilon} N^{1+\epsilon}\left(\frac{1}{q}+\frac{q}{N^k}\right)^{1/k},
    \end{equ}
    if $|\alpha_k-a/q|\leq 1/q^2$ with $(a,q)=1$.
\end{conj}

\begin{prop}\label{lem-app-6}
    Assume that Conjecture \ref{weyl-sum-conj} holds. Let $N>N_0(k,\epsilon)$ and suppose that
    \begin{equ}\label{app-4}
        \Big|\sum_{n=1}^Ne(tn^k+xn)\Big|\geq H>N^{1-\frac{1}{k}+\epsilon}
    \end{equ}
    with $t,x\in [0,1)$.
    Then there are integers $q,r_1,r_2$ such that
    \begin{equ}\label{app-5}
        1\leq q\leq (NH^{-1})^kN^{\epsilon},\quad \gcd(q,r_1,r_2)=1,
    \end{equ}
    and
    \begin{align}
        \Big|t-\frac{r_1}{q}\Big|&\leq q^{-1}(NH^{-1})^kN^{-k+\epsilon},\label{app-6}\\
        \Big|x-\frac{r_2}{q}\Big|&\leq q^{-1}(NH^{-1})^kN^{-1+\epsilon}.\label{app-7}
    \end{align}
\end{prop}
\begin{proof}[Proof of Proposition \ref{lem-app-6}]
    By Dirichlet's theorem (Lemma \ref{lem-app-1}), there exist $0\leq a\leq q_0\leq M=(NH^{-1})^{-k}N^{k-\epsilon}$ with $\gcd(a,q_0)=1$ such that
    \begin{equ}\label{app-8}
         \Big|t-\frac{a}{q_0}\Big|\leq \frac{1}{q_0M}\leq \frac{1}{q_0^2}.
    \end{equ}
    Then Conjecture \ref{weyl-sum-conj}  yields that
    \begin{equ}\label{app-9}
      \Big|\sum_{n=1}^Ne(tn^k+xn)\Big|\lesssim_{k,\epsilon} N^{1+\epsilon/2k}\Big(\frac{1}{q_0}+\frac{q_0}{N^k}\Big)^{1/k}.
    \end{equ}
    Comparing with \eqref{app-4}, we have
    \begin{equ}\label{app-10}
        (NH^{-1})^{-k}N^{-\epsilon/2}\lesssim \frac{1}{q_0}+\frac{q_0}{N^k}.
    \end{equ}
    Note that the term
    \begin{equ}
        \frac{q_0}{N^k}\leq\frac{M}{N^k}=(NH^{-1})^{-k}N^{-\epsilon}
    \end{equ}
    is much smaller than the left-hand side of the inequality \eqref{app-10} above. Hence we have
    \begin{equ}
        q_0\leq (NH^{-1})^kN^{\epsilon/2}.
    \end{equ}
   Plugging $H>N^{1-\frac{1}{k}+\epsilon}$ and $M=(NH^{-1})^{-k}N^{k-\epsilon}=H^{k}N^{-\epsilon}$ into \eqref{app-8}, we have
    \begin{equ}\label{app-13}
        \Big|t-\frac{a}{q_0}\Big|\leq q_0^{-1}H^{-k}N^{\epsilon}<\frac{1}{N^{(k-1)\epsilon}q_0}\frac{N}{N^k}.
    \end{equ}

    In order to apply Lemma \ref{lem-app-2}, we also need to give a rational approximation to $x$. By Lemma \ref{lem-app-1} again, there exist integers $1\leq q_1\leq 2k^2$ and $b$ with $\gcd(b,q_{\textcolor{red}{1}})=1$ such that
    \begin{equ}\label{app-14}
        \Big|q_0x-\frac{b}{q_1}\Big|\leq \frac{1}{2k^{2}q_1}.
    \end{equ}
    Let $q=q_0q_1$, $r_1=aq_1$ and $r_2=b$. Then \eqref{app-13} and \eqref{app-14} can be rewritten into
    \begin{equ}
        \Big|t-\frac{r_1}{q}\Big|<\frac{1}{2k^2q}\frac{N}{N^k},
    \end{equ}
    and
    \begin{equ}
         \Big|x-\frac{r_2}{q}\Big|<\frac{1}{2k^2q}\frac{N}{N^1}.
    \end{equ}
    Note that $\gcd(q,r_1)=q_1\cdot\gcd(q_0,a)=q_1$. By Lemma \ref{lem-app-2}, we have
    \begin{align*}
        H\leq \!\sum_{n=1}^N\!e(tn^k+xn)=\frac{1}{q}&\sum_{n=1}^{q}\!e_{q}(r_1n^k+r_2n)\!\int_0^N\!\!e(\beta_1 x^k+\beta_2 x)\,\mathrm{d}x\\
        &+O(q^{1-\frac{1}{k}+\epsilon}q_1^{\frac{1}{k}}),
    \end{align*}
    where $\beta_1=t-r_1/q$ and $\beta_2=x-r_2/q$.
    If $N$ is sufficiently large, the error term
    \begin{equ}
        O(q^{1-\frac{1}{k}+\epsilon}q_1^{\frac{1}{k}})\leq \frac{1}{2}H,
    \end{equ}
    which implies
    \begin{align}
        \frac{1}{2}H\leq \frac{1}{q}\sum_{n=1}^{q}\!e_{q}(r_1 n^k+r_2 n)\!\int_0^N\!\!e(\beta_1 x^k+\beta_2 x)\,\mathrm{d}x.
    \end{align}
Note that \eqref{app-2} and \eqref{app-103} yields that
\begin{equ}
    \int_0^N\!\!e(\beta_1 x^k+\beta_2 x)\,\mathrm{d}x\lesssim N\min\Big\{1,N^{-1}|\beta_1|^{-1/k},N^{-1/k}|\beta_2|^{-1/k}\Big\},
\end{equ}
and
\begin{equ}
    \frac{1}{q}\sum_{n=1}^{q}\!e_{q}(r_1 n^k+r_2 n)=O(q_1^{1/k}q^{-{1/k}+\epsilon})=O_{k}(q^{-1/k}N^{\epsilon}).
\end{equ}
 Then we have
 \begin{equ}
     \max\{q, N|q\beta_2|,N^k|q\beta_1|\}\lesssim (NH^{-1})^k N^{k\epsilon}.
 \end{equ}
\end{proof}
This lemma is an analogy of Proposition \ref{stru-3}. Similarly, by introducing the definitions of power-free and power-full numbers, we can also get a refined rational approximation  and a super-level-set estimate  corresponding to Proposition \ref{refinement-structure} and Lemma \ref{refinement-level-set} respectively with $D= k$ instead of $\min\{2^{k-1},k(k-1)\}$. That is, for any $A>N^{1-\frac{1}{k}+\epsilon}$, we have the level set estimate
\begin{equ}
 \Big|\Big\{x\in\mathbb{T}:\sup_{0<t<1}\big|\sum_{n=1}^N\!e(tn^k+xn)\big|  >A     \Big\}\Big|\lesssim N^{k+\epsilon}A^{-(k-1)}.
\end{equ}
Then by Lemma \ref{level-set-F} with $a=k+\epsilon$, $b=k+1$ and $M>N^{1-\frac{1}{k}+\epsilon}$ to be determined later, we obtain
 \begin{equ}
  \bigg\|\sup_{0<t<1}\Big|\sum_{n=1}^{N}e^{2\pi i(nx + n^{k}t)}\Big|\bigg\|_{L^{p}(\T)} \lesssim N^{\epsilon}(M+N^{\frac{k}{p}}M^{1-\frac{k+1}{p}}+N^{1-\frac{1}{p}}).
  \end{equ}
  Take $M=N^{1-\frac{1}{k+1}}$, we get the optimal upper bound
   \begin{equ}
  \bigg\|\sup_{0<t<1}\Big|\sum_{n=1}^{N}e^{2\pi i(nx + n^{k}t)}\Big|\bigg\|_{L^{p}(\T)} \lesssim N^{\frac{1}{2}+s_{p,k}+\epsilon}.
  \end{equ}
where for $k\ge3$
\begin{equation*}
  s_{p,k} =
 \begin{cases}
 \frac{1}{2}-\frac{1}{k+1}, \quad &\textup{if}\,\,\,1\le p\le k+1,\\
 \frac{1}{2}-\frac{1}{p}, \quad &\textup{if}\,\,\,p>k+1.
 \end{cases}
\end{equation*}

\end{document}